\documentclass[11pt,a4paper]{article}

\usepackage[margin=1in]{geometry}
\usepackage{amsmath, amssymb, amsthm}
\usepackage{mathrsfs}
\usepackage{graphicx}
\usepackage{hyperref}
\usepackage{xcolor}
\usepackage{fancyhdr}
\usepackage{mathtools}
\usepackage{amssymb,bbm}
\usepackage{amsthm, mathtools, mathrsfs, calligra, unicode-math}
\usepackage{blkarray, bigstrut}
\usepackage{graphicx}
\usepackage{tikz} 
\usepackage{program}
\usepackage{stmaryrd}
\usepackage{hyperref}
\definecolor{titleblue}{RGB}{0,70,140}

\hypersetup{
	colorlinks=true,
	linkcolor=blue,
	citecolor=blue,
	urlcolor=blue
}

\newtheorem{theorem}{Theorem}[section]
\newtheorem{lemma}[theorem]{Lemma}
\newtheorem{proposition}[theorem]{Proposition}
\newtheorem{corollary}[theorem]{Corollary}

\theoremstyle{definition}
\newtheorem{definition}[theorem]{Definition}
\newtheorem{example}[theorem]{Example}

\theoremstyle{remark}
\newtheorem{remark}[theorem]{Remark}

\title{\textbf{\textcolor{titleblue}{Characterization of Normalizer of Lie Superalgebra and its Application to Control Theory}}}

\author{
	Aroonima Sahoo, Kishor Chandra Pati \thanks{Department of Mathematics, National Institute of Technology Rourkela, India-769008}, 
	Tofan Kumar Khuntia\thanks{Department of Mathematics, Harish Chandra Research Institute, Prayagraj, Uttar Pradesh, India-211019}
}
\date{}

\begin{document}
	
	\maketitle
	
	\begin{abstract}
		The dynamical systems having both bosonic and fermionic variables play an important role in the theory of supersymmetry. This article addresses the control problems including both bosonic and fermionic variables on Lie supergroup as the configuration space. Here, the control systems are characterized using the normalizer of Lie subsuperalgebra of left-invariant vector fields in the Lie superalgebra of all smooth vector fields of Lie supergroup. Then, the linear control system is studied in detail and its controllability criterion is proposed along with suitable examples.
	\end{abstract}
	
	\noindent \textbf{Keywords:} Control system, Controllability, Lie supergroup, Lie superalgebras, Vector fields \\
	
	\noindent \textbf{MSC (2020):} 37N35, 58A50, 58C50, 93B05,  93C05
	
	\tableofcontents
	
	\section{Introduction}
	
	 \par Recent times have witnessed significant advancement in the study of dynamical systems that incorporate both fermions and bosons. These kind of systems gives rise to new symmetries, which demand different types (graded) of algebraic structures in order to study them. The ordinary dynamical systems only contain the commuting variables, whereas the ones involving both bosonic and fermionic variables require both commuting (to represent bosons) as well as anticommuting (to represent fermions). The behaviour of such dynamical systems can not be interpreted with the help of ordinary manifold theory. Thus, it needs a more appropriate model space having $\mathbb{Z}_2$-graded geometry. Thus, the supermanifold theory comes into play. For example, one can consider the case of KdV equations \cite{das1990zero}; 

	the ordinary KdV equation is derived from the zero-curvature condition associated with the Lie algebra $\mathfrak{sl}(2, \mathbb{R})$ whereas the super Kdv equation is related with zero-curvature condition associated with the Lie superalgebra $\mathfrak{osp}(2 \vert 1)$. The available literature substantiates the importance of control problems related to Kdv equation in ordinary space. So, this inspires us to pursue same studies for the super case. With this motivation, we deal with the control system on Lie supergroup and try to find the controllability condition for the linear case.\\

	\par In 1951, Schwinger introduced the anticommuting numbers  \cite{schwinger1951theory} for the first time and in 1966, Berezin \cite{berazin2012method, berezin1970lie} developed some mathematics using these anticommuting elements. His theories took shape with the development of supersymmetry theories (Wess and Zumino, 1974 \cite{ferrara1974supergauge}; Arnowitt, 1975 \cite{arnowitt1975superfield}). Firstly, the Schr{\"o}dinger equations were extended to superspace  \cite{delbourgo1989anharmonic, de2008schrodinger}. Later, many other problems were formulated in the graded space such as quantum Kepler problem \cite{zhang2008orthosymplectic}, the quantum harmonic oscillator \cite{macfarlane1991quantum}, quantum (an-)harmonic oscillator\cite{de2007hermite, dunne1988negative}, super KdV equation \cite{das1990zero}, Witten's dynamics \cite{volkov1986hamiltonian, soroka1989hamilton} to name some. Thus, the supermanifold theory is the most suitable space to study these dynamical systems and this space can be viewed as an extension of classical theory which also involves anticommuting co-ordinates. Thus, this field is very useful in the studies of supersymmetry  (\cite{buchbinder1998ideas, deligne1999quantum, varadarajan2004supersymmetry, gates2001superspace}), supergravity ( \cite{freedman1976properties}), superfield theory ( \cite{strathdee1987super, salam1994selected}) and many more. There has been several approaches made for supermanifold theory such as concrete approach by Rogers \cite{rogers2007supermanifolds}, De Witt \cite{dewitt1992supermanifolds}; sheaf theorhetic approach by Leites \cite{leites1980introduction}, Kostant \cite{kostant1977graded} to name some. \\
	
	
	\par The dynamics of a control system $\Sigma$ on a Lie supergroup $\mathcal{G}$ can be given by the vector field on $\mathcal{G}$ parameterized by some control vector $u$. At the outset, the control systems were studied on ordinary Euclidean space $\mathbb{R}^n$ until 1972, when Brockett \cite{brockett1972system} discovered the existence of control systems on Lie group. Subsequently, Jurdjevic and Sussmann  \cite{jurdjevic1972control, sussmann1972controllability} produced many striking results that reshaped the study of non-linear control systems and this led to the development of geometric control theory by  Sachkov \cite{sachkov2009control, sachkov2022introduction}, Agrachev \cite{agrachev2013control} and many others. Meanwhile, Markus \cite{markus2006controllability} studied a particular case of control theory which is the linear control system on matrix Lie groups in 1980, and both Ayala and Tirao \cite{ayala1999linear} generalized this in 1999. The aim of this work is to extend the notion of control system from Lie groups to the case of Lie supergroup by considering a generalized control system characterized by the normalizer. Furthermore, the linear control system on Lie supergroup is disscused in details establishing its controllability condition.\\

	\par  Here, the control system on Lie supergroup $\mathcal{G}$ is proposed using the normalizer of the corresponding Lie superalgebra $\mathfrak{g}$ in the Lie superalgebra of smooth vector fields of $\mathcal{G}$. This characterization helps to generalize the affine control system on a Lie supergroup. the set of smooth vector fields $Vec(\mathcal{G})$ of a Lie supergroup $\mathcal{G}$ forms a Lie superalgbera and the set of left-invariant vector fields $\mathfrak{g}$ forms a Lie subsuperalgebra which is isomorphic to the corresponding Lie superalgbera $\mathfrak{g}$ of the Lie supergroup $\mathcal{G}$ identified as the tangent space $T_e(\mathcal{G})$ at identity $e$ of $\mathcal{G}$. The normalizer $\textnormal{norm}_{Vec(\mathcal{G})} (\mathfrak{g})$ of Lie subsuperalgebra $\mathfrak{g}$ in the Lie superalgebra $Vec(\mathcal{G})$ can be identified with $\mathfrak{g} \otimes \textnormal{Sder}(\mathfrak{g})$, which gives rise to the result that a vector field in the normalizer can be expressed as the comboination of a left-invariant vector field and a linear vector field because a linear super derivation of a Lie superalgebra gives rise to a linear vector field on the corresponding Lie supergroup. \\
	
	\par This article is organized as follows. Section \ref{sec2} contains the necessary notations and existing concepts to understand the foundation of our study. Here, we give an insight into supermanifold, Lie supergroup and Lie superalgebra along with the required theorems. Section \ref{sec3} briefly explains the motivation behind taking the notion of normalizer as it widely covers many types of control systems. Section \ref{sec4} provides detailed analysis of linear control systems on Lie supergroup and authenticates the generalization of linear system from superspace to Lie supergroup. Section \ref{sec5} gives some conclusive remark.
	
	\section{Preliminaries} \label{sec2}

	In this section, the concepts of supermanifold, Lie supergroups and Lie superalgebra and other related terms and definitions are discussed. Most of the definitions are followed from Roger \cite{rogers2007supermanifolds} and DeWitt \cite{dewitt1992supermanifolds}. There are also other approaches available in the literature using sheaf theory such as Varadrajan \cite{varadarajan2004supersymmetry}, Carmeli et. al. \cite{carmeli2011mathematical} and many others.

	The term super reflects the $\mathbb{Z}_2$ gradation of the space. The basic difference from the ordinary space is that the underlying field is replaced by a supercommutative algebra known as Grassmann algebra $\mathbb{K}_{S[L]}$, which is an associative algebra over the field $\mathbb{K}$ with unit having generators $ \xi_1,\cdots,\xi_L $ satisfying anticommuting property i.e.,
	$$\xi_i \xi_j=-\xi_j\xi_i \textnormal{ or equivalently } \xi_i^2=0,~~i,j=1,\cdots,L.$$

	The $(m \vert n)$-dimensional superspace over the Grassmann algebra $\mathbb{R}_{S[L]}$ is given by
	$$\mathbb{R}_{S[L]}^{m\lvert n}=\underbrace{\mathbb{R}_{S[L,0]} \times \cdots \times \mathbb{R}_{S[L,0]}}_\text{m copies} \times
	\underbrace{\mathbb{R}_{S[L,1]} \cdots \times \mathbb{R}_{S[L,1]}}_\text{n copies}.$$
	
	The superspace $\mathbb{R}_{S[L]}^{m\lvert n}$ acts as the counterpart in super analysis as  $\mathbb{R}^n$ in classical analysis and its elements are of the form
	\begin{align*}
		&X=(x;\theta)= (x_1, \cdots, x_n; \theta_1, \cdots, \theta_m),\\
		\text{where } &x_i \in \mathbb{R}_{S[L,0]}, ~i=1, 2, \cdots, m \text{ and } \theta_j \in \mathbb{R}_{S[L,1]}, ~j=1,2,\cdots, n.
	\end{align*}

	For the case of infinite number of generators, the subscript `L' is omitted. Next we discuss an important map called the body map which is a unique algebra homomorphism defined by 
	\begin{align*}
		\epsilon : \mathbb{R}_{S} &\rightarrow \mathbb{R}\\
		\sum_{\lambda \in M_L} x_\lambda \xi_\lambda &\mapsto x_{\emptyset}=x_{\scriptscriptstyle B}.
	\end{align*}
	
	\noindent An extension of the body map that projects the superspace to ordinary Euclidean space is 
	\begin{align*}
		\epsilon_{m,n} : \mathbb{R}_S^{m\vert n} &\rightarrow \mathbb{R}^m \\ (x_1,\cdots , x_{m};\theta_1,\cdots , \theta_n) & \mapsto (\epsilon(x_1),\cdots ,\epsilon(x_m)) =(x_{1_{\scriptscriptstyle B}}, \cdots , x_{m_{\scriptscriptstyle B}}).
	\end{align*}
	
	Here, we follow the DeWitt topology, according to which a subset $\mathcal{V}$ is open in $\mathbb{R}_S^{m\vert n}$ if and only if $\epsilon_{m,n}(V)$ is open in $\mathbb{R}^m$. With the domain defined, one can move forward to the supermanifold concept. A $(m\vert n)$-dimensional supermanifold $\mathcal{M}$ can be defined as a topological space which is locally super Euclidean i.e., every point has a neighborhood that is homeomorphic to an open subset of  superspace $\mathbb{R}_{S}^{m\vert n}.$

	A Lie supergroup $\mathcal{G}$ is a smooth (in $H^\infty$ sense)  supermanifold which has a group structure. Here the transition maps taken are $H^\infty$ instead of $G^\infty$ because the former one helps to show the one-one correspondence between the concrete approach and the sheaf-theoretic approach \cite{rogers2007supermanifolds}. Here we mainly consider the matrix Lie supergroup and some important examples are discussed in section \ref{sec4.3}. For more details on the Lie supergroup, one can refer to Rittenberg \cite{rittenberg2005guide, rittenberg1978elementary} and Berezin \cite{berezin1981group}.

	Analogous to Lie algebra we have the tangent space at identity $T_e(\mathcal{G})$ of a Lie supergroup is isomorphic to the algebra of left-invariant vector fields $\mathcal{L(G)}$, which we denote as Lie superalgebra $\mathfrak{g}$. This is a super vector space with a bilinear map $\llbracket \cdot , \cdot \rrbracket: \mathfrak{g} \times \mathfrak{g} \rightarrow \mathfrak{g}$ satisfying  graded skew-symmetry and graded jacobi-identity. For more details on Lie superalgebra, one can refer to Kac  \cite{kac1977lie}, Frappat et. al. \cite{frappat2000dictionary} and Musson \cite{musson2012lie}.

	\section{Characterization of Normalizer}\label{sec3}
	
	Here, in order to define the dynamics of a control system, the concept of normalizer is used. The basic idea behind it is that a vector field $X$ on $\mathbb{R}^{m\vert n}$ is linear if it is of the form $X_p=A(p)+q$ where $A:\mathbb{R}^{m\vert n} \to \mathbb{R}^{m\vert n} $ is a graded-preserving linear map and $q$ is a fixed supervector of $\mathbb{R}^{m\vert n}$. It can be observed that for any constant super vector field $Y$, the superbracket $\llbracket X,Y \rrbracket$ is a constant vector field. So, the set of all linear super vector fields on $\mathbb{R}^{m\vert n}$ can be thought as the set of normalizer of Lie superalgebra of $\mathbb{R}^{m\vert n}$ in the Lie superalgebra of all $H^\infty$ super vector fields on $\mathbb{R}^{m\vert n}$. So, one may generalize the linear systems of superspace to Lie supergroup using the concepts of normalizer.\\

	Let the set of all $H^\infty$ vector fields on the Lie supergroup $\mathcal{G}$ is denoted as $Vec(\mathcal{G})$ and the set of left-invariant vector fields is denoted as $\mathfrak{g}$. Then the normalizer of Lie subsuperalgebra $\mathfrak{g}$ in the Lie superalgebra $Vec(\mathcal{G})$ is given by
	$$ \textnormal{norm}_{Vec(\mathcal{G})}(\mathfrak{g}) :=\Big\{ X \in Vec(\mathcal{G}) ~ \Big\vert ~  \llbracket X,Y \rrbracket \in \mathfrak{g},
	\textnormal{ for all } Y \in \mathfrak{g}  \Big\}. $$
	
	It is easy to see that $ \textnormal{norm}_{Vec(\mathcal{G})} (\mathfrak{g})$ is a Lie subsuperalgebra of $ Vec(\mathcal{G})$ and the  adjoint map $\textnormal{ad}: \textnormal{norm}_{Vec(\mathcal{G})} (\mathfrak{g}) \rightarrow \textnormal{Sder}(\mathfrak{g})$ defined as $X \mapsto \textnormal{ad}(X)_{\mathfrak{g}}$ is a Lie superalgebra homomorphism, $\textnormal{Sder}(\mathfrak{g})$ is the set of all superderivations of $\mathfrak{g}$. Moreover, the kernel of this adjoint map is the centralizer $\mathfrak{z}(\mathfrak{g})$ of $\mathfrak{g}$ in $ Vec(\mathcal{G})$.
	
	\begin{lemma}
	For a connected Lie supergroup $\mathcal{G}$ the centralizer $\mathfrak{z}(\mathfrak{g})$ consists of all the right-invariant vector fields on $\mathcal{G}$.
	\end{lemma}
	
	\begin{proof}
	Two vector fields commute if and only if their local flows commute. For $X\in \mathfrak{g}$ a left-invariant vector field its flow $ \textnormal{exp}(tX)$ satisfies $p\mapsto p\textnormal{exp}(tX)$ for all $p\in \mathcal{G}$. Thus, for all $X\in \mathfrak{g},p\in \mathcal{G},s,t \in \mathbb{R}^{1\vert 0}$ it is obtained that
	$$ Y\in \mathfrak{z}(\mathfrak{g}) \iff \psi_s^Y (p~ \textnormal{exp}(tX)) =\psi_s^Y(p)~\textnormal{exp}(tX). $$
	This gives $ \psi_s^Y (p q)=\psi_s^Y (p)q,$ for all $q\in \mathcal{G}$, which is the property of right-invariant vector field. Conversely, if $Y$ is right-invariant, then
	$$ \dfrac{d}{ds} \biggr\rvert \psi_s^Y (p q)=\dfrac{d}{ds} \biggr\vert \psi_s^Y (p) q=Y_{\psi_s^Y(p)q},$$
	and $\psi_0^Y (p)=\psi_0^Y(pq)$. Thus, $\psi_s^Y (p)$ and $\psi_s^Y(pq)$ are solution of same ODE with initial condition $t=0$ and by the uniqueness of solution, they must coincide.
	\end{proof}
	
	It is easy to check that $\mathscr{H}=\{ X \in \textnormal{norm}_{Vec(\mathcal{G})}(\mathfrak{g}) \vert X_e=0 \}$ is a Lie subsuperalgebra. Since the centralizer $\mathfrak{z}(\mathfrak{g})$ is a Lie superalgebra ideal in $\textnormal{norm}_{Vec(\mathcal{G})}(\mathfrak{g})$, it is obtained that $\llbracket\mathfrak{z}(\mathfrak{g}),\mathscr{H}\rrbracket=\mathfrak{z}(\mathfrak{g})$ and since the only right-invariant vector field which vanishes at a point is zero vector field, it is obtained that $\mathscr{H} \cap \mathfrak{z}(\mathfrak{g}) =\{0\}$. Therefore, the direct sum $\mathfrak{z}(\mathfrak{g}) \oplus \mathscr{H}$ is well-defined.
	
	\begin{theorem}
	If $\mathcal{G}$ is a connected Lie supergroup, then  $ \mathfrak{z}(\mathfrak{g}) \oplus \mathscr{H}$ and  $\textnormal{norm}_{Vec(\mathcal{G})}(\mathfrak{g})$ are isomorphic.
	\end{theorem}
	\begin{proof}
	Let us consider $X \in \textnormal{norm}_{Vec(\mathcal{G})}(\mathfrak{g})$ and $Z \in \mathfrak{z}(\mathfrak{g})$, a right-invariant vector field satisfying $X_e=Z_e$. Then it is obtained that $(X-Z)_e=0$ that is $X-Z \in \mathscr{H}$. And as  $Z \in \mathfrak{z}(\mathfrak{g})$, which is the kernel of the map $X \mapsto \textnormal{ad} (X)_{\mathfrak{g}} $, and hence $X\in \textnormal{ad} (X)_{\mathfrak{g}}=\textnormal{ad} (X-Z)_{\mathfrak{g}}$. Thus, denoting $Y=X-Z$, any element $X \in \textnormal{norm}_{Vec(\mathcal{G})}(\mathfrak{g})$ can be expressed as $X=Z+Y\in \mathfrak{z}(\mathfrak{g}) \oplus \mathscr{H}$, which proves the case. 
	\end{proof}
	
	\begin{theorem} \label{thmnorm}
	Consider the map $\Phi : \textnormal{norm}_{Vec(\mathcal{G})} (\mathfrak{g}) \rightarrow \mathfrak{g} \otimes \textnormal{Sder} (\mathfrak{g}) $ defined by $\Phi(Z+Y)=(Z',\textnormal{ad} (Y)_{\mathfrak{g}}) $ for all $Z\in \mathfrak{z}(\mathfrak{g})$ and all $Y\in \mathscr{H}$, where $Z'$ is the left-invariant vector field corresponding to $Z$. Then,
	\begin{enumerate}
		\item if $\mathcal{G}$ is a connected Lie supergroup, then $\Phi$ is an injective homomorphism of Lie superalgebras.
		\item if $\mathcal{G}$ is a simply connected Lie supergroup, then $\Phi$ is surjective and hence an isomorphism of Lie superalgebras.
	\end{enumerate}
	\end{theorem}
	
	\begin{proof}
	\begin{enumerate}
		\item Let $Z$ be the right-invariant vector field on $\mathcal{G}$ and $Z'$ be the left-invariant vector field satisfying $Z_e=-Z_e'$. Then the linear map $Z \mapsto Z': \mathfrak{z}(\mathfrak{g}) \mapsto \mathfrak{g}$ is an isomorphism. Thus, by construction of $\Phi$ is an injective Lie superalgebra homomorphism.
		
		\item As simply connectedness implies connectedness,  it is enough to show surjectivity. Consider the fact that for a simply connected Lie supergroup $\mathcal{G}$ there exists a unique automorphism $\phi_t$ of $\mathcal{G}$ such that $(d \phi_t)_e = \textnormal{exp}(tD)$, where $D \in \textnormal{Sder}(\mathfrak{g})$ and $\textnormal{exp}(tD)$ is an automorphism of $\mathfrak{g}$. To show $X \in \textnormal{norm}_{Vec(\mathcal{G})}(\mathfrak{g})$ is the vector field associated with the flow $\phi_t^X(g),~ p \in \mathcal{G}$, by taking $Y \in \mathfrak{g}$ one gets
		\begin{align*}
			\llbracket X,Y \rrbracket_p =& \dfrac{d}{dt} \biggr\rvert_{t=0} d\phi_{-t}^X dL_{\phi_{-t}^X (p)}Y_e \\
			=& \dfrac{d}{dt} \biggr\rvert_{t=0} dL_p d \phi_{-t}^X\\
			=& dL_p \dfrac{d}{dt} \biggr\rvert_{t=0} \textnormal{exp} (-tD) Y_e\\
			=&-(DY)_p
		\end{align*}
	\end{enumerate}
	Thus, $X \in \textnormal{norm}_{Vec(\mathcal{G})}(\mathfrak{g})$ and it is obtained $\textnormal{ad}(X)_{\mathfrak{g}}=-D$ and $X_e=\dfrac{d}{dt} \biggr\rvert_{t=0} \phi_t^X(e)=0$. Hence, the proof is done.
	\end{proof}

	\begin{remark}
	By the theorem \ref{thmnorm}, the normalizer depends on the set of superderivation and the structure of $\textnormal{Sder}(\mathfrak{g})$ may vary somewhere between set of inner superderivation $Inn(\mathfrak{g})$ to $End(\mathfrak{g})$. For example,	taking $\mathcal{G}=\mathbb{R}^{n\vert m}$, it is found that any linear transformation $T:\mathbb{R}^{n\vert m} \rightarrow \mathbb{R}^{n\vert m}$ is a derivation.  \\
    \end{remark}
	
  \section{Control System using normalizer}\label{sec4}
  
  In this section, a generalized form of control system on Lie supergroup is proposed using the normalizer of the corresponding Lie superalgebra.  The idea of taking normalizer began by Ayala and Tirao \cite{ayala1999linear} in the process of generalizing the linear control system to Lie group. The intuition was simple that a constant vector   determines a left-invariant vector field on the usual Euclidean space and the ordinary Lie bracket of a linear vector field  with a constant vector field is always constant on  usual Euclidean space. However, it is seen that this generalization covers a lot of control systems. Using those results, here we propose the control system on Lie supergroup, which is useful in tackling the dynamical systems consisting of both fermionic and bosonic variables.

  \begin{definition}
  	The state equation of a control system on a Lie supergroup $\mathcal{G}$ is given by
  	\begin{equation} \label{linlie}
  		\dot{x}(t)= X(x(t))+\sum_{i=1}^{k} u_i(t)Y^i(x(t))+\sum_{j=1}^{l} \nu_j(t)\Xi^j(x(t))
  	\end{equation} 
  	$x\in \mathcal{G}$; a connected Lie supergroup ,\\
  	$X \in  \textnormal{norm}_{Vec(\mathcal{G})}(\mathfrak{g})$; the normalizer of  $\mathfrak{g}$ in the Lie superalgebra $Vec(\mathcal{G})$,\\
  	$u=(u_1,~u_2,\cdots , u_{k};\nu_1,\cdots,\nu_{l}) $ ; piece-wise constant function in $\mathbb{R}^{k \vert l}$ and\\
  	$Y^i, \Xi^j \in \mathfrak{g}$; set of left-invariant vector fields  of  $\mathcal{G}$. 
  \end{definition}

  The beauty of taking such a model problem using the normalizer is that, this covers a wide range of control systems such as
  
  \begin{enumerate}
  	\item A linear control system on superspace is given by
  	\begin{equation*}
  		\dot{x}(t)=Ax(t)+\sum_{i=1}^{k} u_i(t)b^i+\sum_{j=1}^{l} \nu_j(t)\beta^j,
  	\end{equation*}
  	where $x(t) \in \mathcal{G}=\mathbb{R}^{m\vert n},~A\in \textnormal{Sder}(\mathfrak{g})=\mathfrak{gl}_{m\vert n}(\mathbb{R}),~ b^i,\beta^j \in \mathfrak{g} = \mathbb{R}^{m\vert n}.$\\
  	
  	\item A bilinear control system on superspace is given by
  	\begin{equation*}
  		\dot{x}(t)=Ax(t)+\sum_{i=1}^{k} u_i(t)B^ix(t)+\sum_{j=1}^{l} \nu_j(t)\Xi^jx(t),
  	\end{equation*}
  	where $x(t) \in \mathcal{G}=\mathbb{R}^{m\vert n},~A,B^i,\Xi^j\in \textnormal{Sder}(\mathfrak{g})=\mathfrak{gl}_{m\vert n}(\mathbb{R}).$\\
  	
  	\item A left-invariant control system on Lie supergroup is given by
  	\begin{equation*}
  		\dot{x}(t)=Y(x(t))+\sum_{j=1}^{k+l} u_j(t) Y^j(x(t)),
  	\end{equation*}
  	where $x(t) \in \mathcal{G},~Y,Y^j\in \mathfrak{g}.$\\
  	
  	\item A linear control system on Lie supergroup is given by
  	\begin{equation*}
  		\dot{x}(t)=X(x(t))+\sum_{i=1}^{k} u_iY^i(x(t))+\sum_{j=1}^{l} \nu_j\Xi^j(x(t)),
  	\end{equation*}
  	where $x(t) \in \mathcal{G},~X\in \textnormal{Sder}(\mathfrak{g}), ~Y^j\in \mathfrak{g}.$
  \end{enumerate}

  The first type of system is already dealt with in \cite{sahoo2024controllability} and here we focus on the fourth type.

  \subsection{Linear Control System on Lie Supergroup}\label{sec4.1}

  \par While real-world systems often exhibit nonlinearities, understanding linear systems provides a crucial foundation for tackling more complex nonlinear control problems. This section gives the expression for linear control system on a Lie supergroup. The system becomes a linear system when the drift vector field is linear.  Using the normalizer characterization, one can observe from the theorem \ref{thmnorm} that $X\in  \textnormal{norm}_{Vec(\mathcal{G})}(\mathfrak{g})$ is linear if $X_e=0.$ \\
  
  \par The linear control system on Lie Supergroup is dealt in \cite{sahoo2025controllability} taking normalizer notion but the system considered there involves only the even left-invariant vector fields. Here, we propose a more general version that also involves the odd left-invariant vector fields and it truly generalizes the linear system on superspace.

  \begin{definition}
  A linear control system on a Lie supergroup is determined by $\Sigma = (\mathcal{G}, \mathfrak{D})$, where the state-space $\mathcal{G}$ is a real finite-dimensional Lie supergroup with Lie superalgebra $\mathfrak{g}$ and the dynamic $\mathfrak{D}$ is given by the family of even vector fields associated with $\Sigma$, i.e.,
  $$ \mathfrak{D}=\left\lbrace X+ \sum_{i=1}^{k} u_i Y^i +\sum_{j=1}^{l} \nu_j\Xi^j  \Bigg\vert u=(u;\nu) \in \mathbb{R}_S^{k \vert l} \right\rbrace .$$
  Here, $u=(u_1,~u_2,\cdots , u_{k};\nu_1,\cdots,\nu_{l}): \left[ 0,\infty \right) \rightarrow \mathbb{R}^{k \vert l}$ is the input function from the class of admissible controls $\mathcal{U}$, the drift vector field $X$ lies in the set of the normalizer of $\mathfrak{g}$ in the Lie superalgebra $Vec(\mathcal{G})$ of all $H^{\infty}$ vector fields on $\mathcal{G}$ and the control vectors $Y^i,\Xi^j$ are the even and odd left-invariant vector fields of $\mathcal{G}$. 
  So, the linear system on superspace is expressed as the family of differential equations on $\mathcal{G}$ of the form
  \begin{equation} \label{linlie}
  	\dot{x}(t)= X(x)+\sum_{i=1}^{k} u_i(t)Y^i(x)+\sum_{i=1}^{l} \nu_i(t)\Xi^i(x),
  \end{equation} 
  \noindent where $x$ lies in a connected Lie supergroup $\mathcal{G}$ and others are defined above.
  \end{definition}
  
  
  \begin{remark}
  Observe, when the Lie supergroup is taken to be superspace $\mathcal{G}=\mathbb{R}_S^{m\vert n}$, then the system reduces to the required form of linear control system, 
  \begin{align*}
  	\dot{x}(t)&= Ax+\sum_{i=1}^{k} u_ib^i+\sum_{i=1}^{l} \nu_i\beta^i\\
  	&= Ax+Bu.
  \end{align*} 

  \end{remark}

  From the linear control system on super Euclidean space, one can observe that linear vector field $\mathcal{A}$ generates the flow $\{\textnormal{exp}(t\mathcal{A}):t\in \mathbb{R}^{1\vert 0}\}$ which is a $(1\vert 0)-$parameter subgroup of automorphism group $\textnormal{Aut}(\mathbb{R}^{m\vert n})$ of $\mathbb{R}^{m\vert n}$. 
  So, a super vector field on Lie supergroup is defined to be linear if it induces a $(1\vert 0)-$parameter subgroup of automorphism $Aut(\mathcal{G})$ of Lie supergroup $\mathcal{G}$. As the dynamics of the drift vector field are given by the help of normalizer of Lie superalgebra $\mathfrak{g}$ above, the following section aims to show that the proposed drift vector field induces a subgroup of automorphism.

  \subsection{Linear Vector field}
  This section is crucial in developing results on linear vector fields, which confirms that both the definitions taken for a linear vector field on Lie supergroup are equivalent. 
  
  Let $T_e\mathcal{G}$ be the tangent space of $\mathcal{G}$ at its identity $e$. Then, for a  simply connected Lie supergroup $\mathcal{G}$, the map $\textnormal{Aut}(\mathcal{G}) \mapsto \textnormal{Aut}(T_e\mathcal{G}) :  \phi \mapsto (d\phi )_e $ is an isomorphism. The Lie superalgebra corresponding to the Lie supergroup $\textnormal{Aut}(T_e\mathcal{G})$ is $\textnormal{Sder}(T_e\mathcal{G})$. Then Lie superalgebra of $\textnormal{Aut}(\mathcal{G})$ can also be identified with $\textnormal{Sder}(T_e\mathcal{G})$. Taking $\textnormal{Exp}:\textnormal{aut}(\mathcal{G}) \rightarrow \textnormal{Aut}(\mathcal{G})$ to be the corresponding exponential map, it is obtained that 
  \begin{align*}
  (d ~\textnormal{Exp} D)_e &= e^D, \textnormal{ and }\\
  (\textnormal{Exp} D)(\textnormal{exp} X) &= \textnormal{exp}\lbrace (d~ \textnormal{Exp} D)_eX \rbrace =\textnormal{exp} \lbrace e^D(X)\rbrace
  \end{align*}
  holds for any $D \in \textnormal{aut}(\mathcal{G}) \textnormal{ and } X \in T_e\mathcal{G} $.  If $F \in T_e\mathcal{G}$, then denote $\Tilde{F}$ as the corresponding left-invariant vector field on $\mathcal{G}$ satisfying $\Tilde{F}_e=F$ and hence, $T_e\mathcal{G}$ is identified with the Lie superalgebra $\mathfrak{g}$. Now, identify the vector fields on $\mathcal{G}$ with the functions of $\mathcal{G}$ into $T_e\mathcal{G} $. Let $F: \mathcal{G} \rightarrow T_e\mathcal{G}$, then
  $$ \Tilde{F}_x =dL_x(F(x))=\widetilde{F(x)}_x, ~ ~ ~x \in \mathcal{G}$$
  Thus, the map $F \mapsto \Tilde{F}$ establishes an isomorphism between $H^\infty(\mathcal{G}, T_e\mathcal{G})$ and $Vec(\mathcal{G})$.
  
  \begin{lemma}\label{chap4lem3.1}
  Given $Y \in T_e\mathcal{G}$ and $F \in H^\infty (\mathcal{G}, T_e\mathcal{G})$, let $H \in  H^\infty (\mathcal{G}, T_e\mathcal{G}) $ be defined by
  \begin{equation}\label{chap4eqn1}
  	H(x)=\Tilde{Y}_x(F)+ \llbracket Y,F(x) \rrbracket, ~ ~ ~ x \in \mathcal{G} .
  \end{equation}
  
  Then $\Tilde{H}= \llbracket \Tilde{Y}, \Tilde{F} \rrbracket . $ Moreover, for $\Tilde{F} \in norm_{Vec(\mathcal{G})} (\mathfrak{g})$, the following result is obtained

  \begin{equation}
  	\Phi (\Tilde{F})=\begin{cases}
  		\left( -F(e), ~\textnormal{ad} (F(e))-dF_e \right), & \text{if $Y$ is even},\\
  		\left( -F(e), ~\textnormal{ad} (F(e))-(-1)^{\vert F\vert}dF_e \right), & \text{if $Y$ is odd}.
  	\end{cases}
  \end{equation}

  \end{lemma}
  
  \begin{proof}
  By the property of Lie superbracket
  $$ \left( \llbracket \Tilde{Y}, \Tilde{F} \rrbracket f  \right) (x) = 
  \left( \Tilde{Y} \left( \Tilde{F}f \right) \right) (x) - (-1)^{\vert \Tilde{Y} \vert \vert \Tilde{F}\vert}  \left( \Tilde{F} \left( \Tilde{Y}f \right) \right) (x)$$
  Simplify first term
  \begin{align*}
  	\left( \Tilde{Y} \left( \Tilde{F}f \right) \right) (x) 
  	&=   \dfrac{d}{dt} \biggr\rvert_{t=0} \left( \Tilde{F}f \right) (x ~\textnormal{exp} tY) \\
  	&= \dfrac{d}{dt} \biggr\rvert_{t=0} \left( F (\widetilde{x ~\textnormal{exp} } tY) f \right) (x ~\textnormal{exp} tY) \\ 
  	&= \left( \widetilde{(\Tilde{Y}F)(x)}f  \right) (x) + \left( \Tilde{Y} \left( \widetilde{(F)(x)}f  \right) \right) (x)
  \end{align*}
  Simplify second term
  
  $$ \left( \Tilde{F} \left( \Tilde{Y}f \right) \right) (x) 
  = \Tilde{F}_x \left( \Tilde{Y} f \right) \\
  = \left( \widetilde{F(x)} \left( \Tilde{Y} f \right) \right) (x).$$
  
  So, the superbracket
  \begin{align*}
  	\left( \llbracket \Tilde{Y}, \Tilde{F} \rrbracket f \right)_x 
  	&= \left( \widetilde{(\Tilde{Y}F)(x)} f \right) (x) + \left( \llbracket \Tilde{Y}, \widetilde{F(x)} \rrbracket f \right) (x)\\
  	&=  \left( \widetilde{(\Tilde{Y}F)(x)} f \right) (x) + \left( \llbracket \widetilde{ \Tilde{Y}, F(x)} \rrbracket f \right) (x)
  \end{align*}
  Therefore it is obtained that,
  $$  \llbracket \Tilde{Y}, \Tilde{F} \rrbracket_x= \widetilde{(\Tilde{Y}_x(F))}_x+ \llbracket \widetilde{ \Tilde{Y}, F(x)} \rrbracket_x= \left( \widetilde{(\Tilde{Y}_x(F))+ \llbracket  \Tilde{Y}, F(x) \rrbracket} \right)_x$$
  
  For computing $\Phi(\Tilde{F})$, take $\Tilde{F}=Z+X$ with $Z \in \mathfrak{z}(\mathfrak{g})$ and $X \in \mathscr{H}$. Then $\Phi(\Tilde{F}) =\left( Z_e, ~\textnormal{ad} (X)_{\mathfrak{g}} \right)=\left( -F(e), ~\textnormal{ad} (X)_{\mathfrak{g}} \right) $. Now given $Y \in T_e\mathcal{G}$, it is obtained
  $$ \llbracket X, \Tilde{Y} \rrbracket_e=\llbracket \Tilde{F}, \Tilde{Y} \rrbracket_e = -\Tilde{Y}_e(F)-\llbracket Y, F(e) \rrbracket =-dF_e(Y)+ (-1)^{\vert F \vert \vert Y \vert}\llbracket F(e),Y\rrbracket. $$
  Thus, the proof is done.
  
  \end{proof}

  \begin{theorem}
  For $F \in H^\infty (\mathcal{G},T_e\mathcal{G})$, the correspondence vector field $\tilde{F} \in \textnormal{norm}_{Vec(\mathcal{G})}(\mathfrak{g})$ iff for all $Y\in T_e\mathcal{G}$ the following holds
  \begin{equation}\label{chap4eqn2}
  	F(x \textnormal{ exp}Y)=(e^{-adY})F(x)+\left(  \dfrac{1-e^{-adY}}{adY}\right)  (  Y(F) + \llbracket Y,F(e) \rrbracket).
  \end{equation}
  \end{theorem}

  \begin{proof}
  Given $F \in H^\infty (\mathcal{G},T_e\mathcal{G})$, $\tilde{F} \in \textnormal{norm}_{Vec(\mathcal{G})}(\mathfrak{g})$ if and only if $\llbracket Y,F(x) \rrbracket \in \mathfrak{g},~ \forall Y\in T_e(\mathcal{G})$, and by lemma \ref{chap4lem3.1}, this will happen iff
  $$\tilde{Y}_x(F)+ \llbracket Y,F(x) \rrbracket = Y(F)+\llbracket Y,F(e) \rrbracket, \qquad \forall  x\in \mathcal{G}$$
  By taking $B(Y)=Y(F)+\llbracket Y,F(e) \rrbracket$, one can have
  $$ \frac{d}{dt} F(x~ \textnormal{exp} (tY)) = -\llbracket Y, F(x ~\textnormal{exp} (tY))\rrbracket + B(Y) ~~\forall x,t. $$
  Taking $G(t) = F(x~ \textnormal{exp} (tY)) $, one can obtain
  \begin{align*}
  	G^{(1)}(t)&=-adY(G(t))+B(Y)\\
  	G^{(2)}(t)&=(-adY)^2(G(t))-adY(B(Y))\\
  	\vdots\\
  	G^{(n)}(t)&=(-adY)^n(G(t))+(-adY)^{n-1}(B(Y)).
  \end{align*}
  Thus, by Maclaurin series expansion
  \begin{align*}
  	G(t)&=\sum_{n=0}^{\infty}\dfrac{(-adY)^n}{n!}F(x)t^n + \sum_{n=1}^{\infty} \dfrac{(-adY)^{n-1}}{n!}B(Y)t^n \\
  	&= e^{-tadY}f(x)+\left(\dfrac{1-e^{-tadY}}{adY}\right)B(Y).
  \end{align*}
  \end{proof}

  \begin{corollary}
  If $\tilde{F} \in \textnormal{norm}_{Vec(\mathcal{G})}(\mathfrak{g})$, then for all $Y\in T_e(\mathcal{G})$,
  $$ dF_x(\tilde{Y}_x) = \llbracket Y, F(e)-F(x) \rrbracket+dF_e(Y). $$
  \end{corollary}
  \begin{proof}
  This follows from \ref{chap4eqn2} by replacing $Y$ with $tY$ and differentiating it with respect to $t$ at $t=0$. 
  \end{proof}

  \begin{corollary}
  For a connected Lie supergroup $\mathcal{G}$, an element $\tilde{F} \in \textnormal{norm}_{Vec(\mathcal{G})}(\mathfrak{g})$ is determined by a pair $(F(e),dF_e)$.
  \end{corollary}
  \begin{proof}
  Taking $x=e$ in Equation (\ref{chap4eqn2}), one can obtain
  $$	F( \textnormal{ exp}(Y))=(e^{-adY})F(e)+\left(  \dfrac{1-e^{-adY}}{adY}\right)  (  dF_e(Y) + \llbracket Y,F(e) \rrbracket).$$
  Thus, $\tilde{F}$ is determined by $(F(e),dF_e)$ on a neighbourhood $U$ of $e$ . Since, $\mathcal{G}$ is connected Lie supergroup, $U$ can induce the whole $\mathcal{G}$ and hence the result follows for all $x\in \mathcal{G}$.
  \end{proof}
  
  \begin{corollary}
  Let $\mathcal{G}=\mathbb{R}^{m\vert n}$. Then $\tilde{F} \in \textnormal{norm}_{Vec(\mathcal{G})}(\mathfrak{g})$ iff $X_p=A(p)+q$ for all $p\in \mathbb{R}^{m\vert n}$, where $A$ is an even linear operator on $\mathbb{R}^{m\vert n}$ and $q$ is some fixed vector on $\mathbb{R}^{m\vert n}$.
  \end{corollary}
  \begin{proof}
  The tangent space of $\mathbb{R}^{m\vert n}$ at a fixed point, is isomorphic to $\mathbb{R}^{m\vert n}$ itself. Thus, the vector field $\tilde{F}$ remains same as the map $F\in H^\infty (\mathbb{R}^{m\vert n},\mathbb{R}^{m\vert n})$. Now, taking $x=0$ and $Y=p$  in Equation (\ref{chap4eqn2}), one obtains $F(p)=dF_0(p)+F(0)$ and hence the proof is done.
  \end{proof}
  \vspace{0.5cm}
  Thus, it is shown that taking the drift vector field from the normalizer of the Lie supergroup with $X_e=0$ obtains the linear vector field in the superspace. Now, the next part shows that it also generates the required flow as a linear vector field.\\

  Let $Diff(\mathcal{G})$ denote the supergroup of all diffeomorphisms of $\mathcal{G}$.  Let $\beta : \mathcal{G} \otimes \textnormal{Aut}(\mathcal{G}) \mapsto Diff(\mathcal{G}) $ be a map defined by $\beta (x,\Gamma)=R_{x^{-1}}\cdot \Gamma$ such that the following diagram commutes,
  
  \begin{center}
  \begin{tikzpicture}[>=latex]
  	\node (x) at (0,0) {\( \textnormal{norm}_{Vec(\mathcal{G})}(\mathfrak{g}) \)};
  	\node (z) at (0,-2.25) {\( Diff(\mathcal{G}) \)};
  	\node (y) at (6,0) {\( \mathfrak{g} \otimes \mathfrak{aut}(\mathfrak{g}) \)};
  	\node (w) at (6,-2.25) {\( \mathcal{G} \otimes \textnormal{Aut} (\mathcal{G})  \)};
  	\draw[->] (x) -- (y) node[midway,above] {};
  	\draw[->] (x) -- (z) node[midway,left] {$Exp$};
  	\draw[->] (w) -- (z) node[midway,above] {};
  	\draw[->] (y) -- (w) node[midway,right] {$Exp$};
  \end{tikzpicture}
  \end{center}
  
  \begin{lemma}
  The map $\alpha:\mathcal{G} \times \textnormal{Aut}(\mathcal{G}) \to Diff(\mathcal{G})$  is an injective supergroup homomorphism.
  \end{lemma}
  
  \begin{proof}
  Let us consider
  \begin{align*}
  	\beta ((x,\Gamma)(y,\Psi)) &= \beta(x\Gamma(y),\Gamma \Psi)\\
  	&=R_{(x\Gamma(y))^{-1}} \Gamma \Psi \\
  	&= R_{x^{-1}} (\Gamma R_{y^{-1}}\Gamma^{-1})\Gamma \Psi \\
  	&= \beta (x,\Gamma) \beta (y,\Psi).
  \end{align*}
  The penultimate step is due to 
  $$ R_{(x\Gamma(y))^{-1}} = R_{x^{-1}} R_{\Gamma(y)^{-1}} = R_{x^{-1}} (\Gamma R_{y^{-1}}\Gamma^{-1}). $$
  Hence, the homomorphism is proved. 
  Now to show injective, consider $\beta(x,\Gamma)=I$, which implies $e=R_{x^{-1}}\Gamma(e)=x^{-1}$ and hence, $I=\beta(e,\Gamma)=\Gamma$. 
  \end{proof}

  \begin{theorem}
  The  $(1\vert 0)$-parametric group of $\tilde{F}_t$ of diffeomorphisms associated to $\tilde{F} \in \textnormal{norm}_{Vec(\mathcal{G})}(\mathfrak{g})$ is given by
  $$ \tilde{F}_t=\textnormal{Exp}(-t\tilde{F}). $$
  \end{theorem}
  
  \begin{proof}
  Since $Exp(-t\tilde{F})$ is a $(1\vert 0)$-parameter subgroup of diffeomorphism of $\mathcal{G}$, it is enough to find the associated vector field $\tilde{H}$. Let $x(t)=Exp(-t\tilde{F})(x)$, $\Phi (\tilde{F})=(E,\phi)$ and $Exp(-t(E, \phi))=(g(t), \Gamma(t))$, then
  \begin{equation}\label{eqn6}
  	x(t)=\textnormal{Exp}(-t(E,\phi))(x)=(R_{g(t)^{-1}} \Gamma(t))(x)=(\Gamma(t)(x))g(t)^{-1} .
  \end{equation}
  
  Let $\mu : \mathcal{G} \times \textnormal{Aut}(\mathcal{G}) \mapsto \mathcal{G}$ be defined by $\mu(x,\Gamma)=\Gamma(x)$, then
  \begin{align*}
  	\frac{d}{dt} \biggr\rvert_{t=0} \Gamma(t)(x) &=d\mu_{(x,I)} (0,\dot{\Gamma}(0))\\
  	&=d\mu_{(x,I)} (0,-\phi)\\
  	&= 	\frac{d}{dt} \biggr\rvert_{t=0} \mu(x,\textnormal{Exp}(-t\phi))\\
  	&= 	\frac{d}{dt} \biggr\rvert_{t=0} (\textnormal{Exp}(-t\phi))(exp(X))\\
  	&= 	\frac{d}{dt} \biggr\rvert_{t=0} exp(e^{-t\phi}(X))\\
  	&= 	\frac{d}{dt} \biggr\rvert_{t=0} exp_X (-\phi(X)).
  \end{align*}
  Now using Equation (\ref{eqn6}), one may obtain
  $$ \dot{x}(0) = \frac{d}{dt} \biggr\rvert_{t=0} \Gamma(t)(x) +(dL_x)(E) =(dL_x)_e \left(  \dfrac{1-e^{-adX}}{adX}\right) (-\phi(X)+E) .$$
  Then the map $H\in H^\infty (\mathcal{G}, T_e(\mathcal{G}))$ corresponding to $\tilde{H}$ satisfies
  $$ H(exp(X))=\left(  \dfrac{1-e^{-adX}}{adX}\right) (-\phi(X)+E), $$
  which is in the required form of Equation (\ref{chap4eqn2}), Hence, $H(x)=F(x),~\forall x\in \mathcal{G}$ completes the proof.
  \end{proof}
  
  In the next section, the controllability of a linear control system is given in details.

  
  \subsection{Controllability}\label{sec4.2}
  
  This section is devoted to establishing some controllability results on the linear systems having the drift vector field $X$ which is an infinitesimal automorphism of $\mathcal{G}$ i.e., the flow $\{ \phi_t^X \vert ~t \in \mathbb{R}^{1 \vert 0 } \}$ generated by $X \in \mathfrak{aut}(\mathcal{G})$ is a $(1\vert 0)$ parametric subgroup of automorphism $\textnormal{Aut}(\mathcal{G})$ of $\mathcal{G}$. The control vector fields $Y^i $ are left-invariant vector fields on $\mathcal{G}$ and  let the family of flows generated by drift vector field $X$ be
  $$\mathcal{T} =\{ \phi_t^X \vert t\in \mathbb{R}^{1\vert 0} \}\subset \textnormal{Aut}(\mathcal{G}).\\$$

  \noindent Here, we obtain two important structures, one is  the group  
  $$G_{\Sigma} := \{ \phi_{t_1}^{Z^1} \circ \phi_{t_2}^{Z^2} \circ \cdots \circ \phi_{t_r}^{Z^r}  \vert~ Z^i \in \mathcal{D}, ~ t_i \in \mathbb{R}^{1 \vert 0},~ (t_i)_{B} \in \mathbb{R} \} $$
  and the other is the semigroup
  $$S_{\Sigma}:= \{  \phi_{t_1}^{Z^1} \circ \phi_{t_2}^{Z^2} \circ \cdots \circ \phi_{t_r}^{Z^r}  \vert ~Z^i \in \mathcal{D},~ t_i \in \mathbb{R}^{1 \vert 0},~ (t_i)_{B} \geq 0 \} $$
  of diffeomorphisms on $\mathcal{G}$, where $(t_i)_B$ is the body part of $t_i$ and $~ r \in \mathbb{N}$.
  
  \begin{definition}
  If the group $G_{\Sigma} $ acts transitively on the state space $\mathcal{G}$, then the system $\Sigma$ is said to be transitive.
  \end{definition}
  
  \noindent The group $G_{\Sigma} $ acting on $x\in \mathcal{G}$ gives the orbit through $x$ and the importance of the semi group structure is that it forms the reachable set from $x\in \mathcal{G}$.

  \begin{definition}
  The system $\Sigma$ is said to be completely controllable if the reachable set covers the whole state space i.e.,
  $$ S_{\Sigma}(x) =\mathcal{G}, ~\forall x\in \mathcal{G} $$
  \end{definition}

  \noindent The $\Sigma$-orbits $G_{\Sigma}(x),~ x\in \mathcal{G}$ are the integral supermanifolds of the distribution generated by $sLie(\Sigma)$. Analogous to the ordinary case, there exists a one-one correspondence between connected Lie subsupergroups of $\mathcal{G}$ and Lie subsuperalgebras of $\mathfrak{g}$ (and hence with the Lie supermodule $\mathfrak{g}$) through the exponential map.   Let $\mathfrak{h}$ be the Lie subsuperalgebra of $\mathfrak{g}$ generated by the control vectors $C(\Sigma)=\{ Y^1,\cdots,Y^{k}, \Xi^1,\cdots, \Xi^l \}$ of $\Sigma$ i.e.,
  $$ \mathfrak{h}=\textnormal{Span}_{L.S.A.} (C(\Sigma)).$$

  \begin{lemma}\label{ad-invariant}
  If $X$ is an infinitesimal automorphism, then $\mathfrak{g}$ is $\textnormal{ad}^i(X)$-invariant for each $i \geq 1$.
  \end{lemma}
  \begin{proof}
  It is enough to show for the case $i=1$. Let $\phi_t^X,~\psi_s^Y$ are the flows of $X,~Y$ respectively where $Y\in \mathfrak{g}$ and use the expression for Lie superbracket 
  \begin{align*}
  	\llbracket X,Y \rrbracket (x) &= \dfrac{\partial}{\partial t} \biggr\rvert_{t=0} \dfrac{\partial}{\partial s} \biggr\rvert_{s=0} \phi_{-t}^X \circ \psi_s^Y \circ \phi_t^X (x)\\
  	&= \dfrac{\partial}{\partial t} \biggr\rvert_{t=0} \dfrac{\partial}{\partial s} \biggr\rvert_{s=0} \phi_{-t}^X (\phi_t^X (x) \textnormal{exp}(sY)) \\
  	&= - \dfrac{\partial}{\partial s} \biggr\rvert_{s=0} X_{x~\textnormal{exp}(sY)} \\
  	&= - \dfrac{\partial}{\partial s} \biggr\rvert_{s=0} (L_x)_* X_{\textnormal{exp}(sY)}
  \end{align*}
  On the other hand, since $X$ is infinitesimal automorphism, $\phi_t^X(e)=e$ for all $t\in \mathbb{R}^{1\vert 0}$ which results the following,
  \begin{align*}
  	\llbracket X,Y \rrbracket (e) &= \dfrac{\partial}{\partial t} \biggr\rvert_{t=0} \dfrac{\partial}{\partial s} \biggr\rvert_{s=0} \phi_{-t}^X \circ \psi_s^Y(e)\\
  	&= \dfrac{\partial}{\partial t} \biggr\rvert_{t=0} \dfrac{\partial}{\partial s} \biggr\rvert_{s=0} \phi_{-t}^X (\textnormal{exp}(sY))\\
  	&=  -\dfrac{\partial}{\partial s} \biggr\rvert_{s=0} X_{\textnormal{exp}(sY)}
  \end{align*}
  Therefore, we obtain
  $$ \llbracket X,Y \rrbracket (x) = (L_x)_{\ast} \llbracket X,Y \rrbracket (e) $$
  Thus, $\llbracket X,Y \rrbracket$ is a left-invariant vector field in $Vec(\mathcal{G})$i.e., $\llbracket X,Y \rrbracket \in \mathfrak{g}$.  
  \end{proof}

  Now $\langle X \vert \mathfrak{h} \rangle$ is denoted as the smallest ad($X$)-invariant subsuperalgebra of $\mathfrak{g}$ containing $\mathfrak{h}$ whereas $\mathcal{H}$ and $\langle X \vert \mathcal{H} \rangle$ are denoted as the connected Lie supergroup of $\mathcal{G}$ with Lie superalgebras $\mathfrak{h}$ and $\langle X \vert \mathfrak{h} \rangle$ respectively.
  Now, construct the following sequence.
  \begin{align*}
  \mathfrak{h}_0&=\mathfrak{h}\\
  \mathfrak{h}_i &=\mathfrak{h}_{i-1}+\textnormal{ad}^i(X)(\mathfrak{h}), ~ ~ i \in \mathbb{N},\\
  \textnormal{where } \textnormal{ad}^i(X)(\mathfrak{h})&=\{ \llbracket X,\textnormal{ad}^{i-1} (X)(V) \rrbracket \vert V\in \mathfrak{h} \}.
  \end{align*}
  For the finite-dimensional case, we end up having a positive integer $p$ for which the induction terminates and hence, by Lemma \ref{ad-invariant}, it is obtained to be a subsuperalgebra of $\mathfrak{g}$ which also contains $\mathfrak{h}$. Therefore, one can conclude that 
  $$ \mathfrak{h}_p= \langle X \vert \mathfrak{h} \rangle .$$
  
  \begin{proposition}\label{prop1}
  If $~\Sigma=(\mathcal{G}, \mathfrak{D})$ is a linear control system of the form (\ref{linlie}), then 
  $$ \phi_t^X \in \textnormal{Aut} (\langle X \vert \mathcal{H} \rangle), \qquad t\in \mathbb{R}^{1\vert 0}. $$
  \end{proposition}
  \begin{proof}
  Observe the fact that any subspace $V$ of $\mathfrak{g}$ induces a distribution of left-invariant vector fields by left translation
  $$ \Delta_V(x)=(L_x)_*V, \qquad x\in \mathcal{G}. $$
  Consider the distribution $\Delta=\Delta_{\langle X \vert \mathfrak{h} \rangle}$, which is integrable by the Frobenius theorem  \cite{bruzzo1985differential, yagi1988super}. Let $I_\Delta(x)$ denotes the integral manifold of $\Delta$ through the point $x\in \mathcal{G}$, then the $\textnormal{ad}(X)$-invariance of $\Delta$ gives us
  $$ \phi_t^X(I_\Delta(x))=I_\Delta(\phi_t^X(x)) $$
  As $\Delta$ is induced by  subspace $\langle X \vert \mathfrak{h} \rangle$ of left-invariant vector fields, 
  $$I_\Delta (x)=x \langle X \vert \mathcal{H} \rangle.$$
  Thus, it is obtained that
  $$ \phi_t^X(x\langle X \vert \mathcal{H} \rangle)=\phi_t^X(x) \langle X \vert \mathcal{H} \rangle $$
  and in particular, for $x=e$,
  $$ \phi_t^X(\langle X \vert \mathcal{H} \rangle)= \langle X \vert \mathcal{H} \rangle .$$ 
  \end{proof}
  
  Thus, the Lie supergroup $\langle X \vert \mathcal{H} \rangle $ is $\mathcal{T}$-invariant and using this fact, we propose an explicit expression for Lie superalgebra $sLie(\Sigma)$ generated by the system $\Sigma$.

  \begin{theorem}
  The Lie superalgebra associated with the linear control system $\Sigma = (\mathcal{G},\mathfrak{D})$ of the form (\ref{linlie}) can be expressed as
  $$ L(\Sigma) \cong \langle X \vert \mathfrak{h} \rangle \oplus L(\mathcal{T}) \cong L(\langle X \vert \mathcal{H} \rangle \rtimes \mathcal{T}). $$
  \end{theorem}
  
  \begin{proof}
  This proof is divided into two parts. First, the Lie superalgebra structure of the semidirect sum  $\langle X \vert \mathfrak{h} \rangle \oplus L(\mathcal{T})$ is given and in the second part,  the Lie superalgebra structure of $L(\Sigma)$ is given subsequently the isomorphism is established. \\
  
  As Lie supergroup $\langle X \vert \mathcal{H} \rangle$ is $\mathcal{T}$-invariant, one can consider the canonical map
  $$ \rho:\mathcal{T} \rightarrow \textnormal{Aut}(\langle X \vert \mathcal{H} \rangle) $$
  and thus, $\langle X \vert \mathcal{H} \rangle \times \mathcal{T}$ is a subgroup of the Lie supergroup $\mathcal{G} \times \textnormal{Aut}(\mathcal{G}).$ Moreover, $\rho$ induces an analytic action
  \begin{align*}
  	\langle X \vert \mathcal{H} \rangle \times \mathcal{T} &\rightarrow \langle X \vert \mathcal{H} \rangle\\
  	(x,\phi_t^X) &\mapsto \phi_t^X(x).
  \end{align*}
  The semidirect product of $\langle X \vert \mathcal{H} \rangle$ with $\mathcal{T}$ relative to $\rho$ is well-defined w.r.t to following product,  
  $$(x_1,\phi_{t_1}^X)\circ (x_2,\phi_{t_2}^X)=(\phi_{t_1}^X(x_2)(x_1),\phi_{t_1+t_2}^X)$$
  for all $x_1,x_2 \in \langle X \vert \mathcal{H} \rangle$ and $\phi_{t_1}^X,\phi_{t_2}^X \in \mathcal{T}$. Denote this Lie Supergroup by
  $$S=\langle X \vert \mathcal{H} \rangle \rtimes \mathcal{T}.$$
  Furthermore, the derivative of the canonical map gives us
  $$ d\rho: L(\mathcal{T}) \rightarrow \textnormal{Der} (\langle X \vert \mathcal{H} \rangle) \subset \textnormal{Aut}(\langle X \vert \mathcal{H} \rangle). $$
  This forms a homomorphism, in fact, for all $W\in L(\mathcal{T})$
  $$d\rho (W)(Y)=\llbracket W,Y \rrbracket,\qquad \forall Y\in \langle X \vert \mathcal{H} \rangle.$$
  
  Now $d\rho$ induces a Lie superalgebra structure  on$\langle X \vert \mathfrak{h} \rangle \times  L(\mathcal{T}),$ where the Lie superbracket is given as 
  \begin{align*}
  	\llbracket (Y_1,W_1),(Y_2,W_2) \rrbracket =& \bigg(  \llbracket Y_1,Y_2 \rrbracket+d\rho (W_1)(Y_2)- (-1)^{\vert W_2 \vert \vert Y_1\vert} d\rho(W_2)(Y_1), ~\llbracket W_1,W_2 \rrbracket \bigg)
  \end{align*}

  Thus, the semi-direct sum $\langle X \vert \mathfrak{h} \rangle \oplus L(\mathcal{T})$ relative to $d\rho$ is well-defined. Denote the Lie superalgebra $\mathfrak{s}:=\langle X \vert \mathfrak{h} \rangle \oplus L(\mathcal{T})$. \\
  
  In the second part, we consider the Lie superalgebra structure of $L(\Sigma)$. As   
  \begin{align*}
  	L(\Sigma)= \textnormal{Span}_{L.S.A.} \biggr\{  X+ \sum_{i=1}^{k} u_i Y^i + \sum_{j=1}^{l} \nu_j \Xi^j \biggr\rvert ~u \in \mathbb{R}^{k \vert 0}, \nu \in \mathbb{R}^{0 \vert l} \biggr\} ,
  \end{align*}
  we may choose any arbitrary elements like $X+Y^i,~X+\Xi^j$ and find the Lie superbracket
  \begin{align*}
  	\llbracket X+Y^i,~X+Y^j \rrbracket &= \llbracket Y^i,~Y^j \rrbracket+ \llbracket X,~Y^j \rrbracket+\llbracket Y^i,~X \rrbracket,\\
  	\llbracket X+Y^i,~X+\Xi^j \rrbracket &= \llbracket Y^i,~\Xi^j \rrbracket+ \llbracket X,~\Xi^j \rrbracket+\llbracket Y^i,~X \rrbracket,\\
  	\llbracket X+\Xi^i,~X+\Xi^j \rrbracket &= \llbracket \Xi^i,~\Xi^j \rrbracket+ \llbracket X,~\Xi^j \rrbracket+\llbracket \Xi^i,~X \rrbracket.
  \end{align*}
  We observe that $\llbracket \mathfrak{s},\mathfrak{s} \rrbracket $ is contained in $ \langle X \vert \mathfrak{h} \rangle$ due to the supercommutativity of $L(\mathcal{T})$. Now, considering the graded skew symmetry property of Lie superbracket and comparing the  superbrackets of $\mathfrak{s}$ and $L(\Sigma)$, it is clear that both are isomorphic. Hence, the first equality holds.\\
  
  For the second equality, use the fact that the Lie superalgbera of semidirect product of Lie supergroups is isomorphic to the semi direct sum of corresponding Lie superalgberas.

  \end{proof}
  
  Now we give an important theorem for a linear control system to be transitive, which is analogous to the Lie Algebra Rank condition (LARC) in ordinary case.

  \begin{theorem}
  If the system $\Sigma$ is transitive, then it satisfies the Lie Superalgebra Rank Condition (LSARC) i.e.,
  \begin{align*}
  	\textnormal{dim}\Big( \textnormal{Span}_{L.S.A.} \big \{ C(\Sigma), ad(X)(C(\Sigma))\big\} \Big) =\textnormal{dim}(\mathcal{G})  
  \end{align*}
  or equivalently,
  \begin{align*}
  	\textnormal{dim}\Big( \textnormal{Span}_{L.S.A.} \big \{ & Y^{i_1},\Xi^{j_1}, ad^{i_2}(X)(Y^{i_1}), ad^{j_2}(X)(\Xi^{j_1}) \big \vert \\
  	& 0 \leq i_2,j_2 \leq p, 1\leq i_1\leq k, 1\leq j_1 \leq l \big\} \Big) =\textnormal{dim}(\mathcal{G}).  
  \end{align*}
  \end{theorem}

  \begin{proof}
  If the system is transitive, using proposition (\ref{prop1}) we get
  $$G_\Sigma (e)=\langle X \vert \mathcal{H} \rangle=\mathcal{G}$$ 
  and the integral manifold of $L(\Sigma)$ through $x\in \mathcal{G}$ coincides with the orbit of $x$. Again  for every $x\in \langle X \vert \mathcal{H} \rangle $
  $$ \dfrac{d}{dt} \biggr\rvert_{t=0} \phi_t^X(x) \in T_x \langle X \vert \mathcal{H} \rangle$$
  or equivalently
  $$ X_x \in (L_x)_* \langle X \vert \mathcal{H} \rangle.$$
  So, there exists some positive integer $p$ less than the dimension $n$ of the $\mathcal{G}$ such that
  $$ \langle X \vert \mathfrak{h} \rangle = \mathfrak{h_{p-1}} +ad^p (X)(\mathfrak{h})=\mathfrak{g} $$
  and hence it follows the LSARC.    
  \end{proof}

  The above theorem provides a necessary condition for controlability whereas for sufficient case we require more technicalities, which is given below.

  \begin{theorem}
  Let $\mathcal{G}$ be a connected Lie supergroup and let $\Sigma=(\mathcal{G},\mathfrak{D})$ be a transitive linear control system whose drift vector field $X$ is an infinitesimal automorphism. Then $\Sigma$ is locally controllable if it satisfies the super ad-rank condition i.e.,
  $$ \textnormal{dim}\Big( \textnormal{Span} \big \{ C(\Sigma), ad(X)(C(\Sigma))\big\} \Big) =\textnormal{dim}(\mathcal{G}) .$$
  or equivalently,
  \begin{align*}
  	\textnormal{dim}\Big(\textnormal{Span} \big \{ & Y^{i_1},\Xi^{j_1}, ad^{i_2}(X)(Y^{i_1}), ad^{j_2}(X)(\Xi^{j_1})\big \vert \\
  	& 0 \leq i_2,j_2 \leq p, 1\leq i_1\leq k, 1\leq j_1 \leq l \big\} \Big) =\textnormal{dim}(\mathcal{G})     .     
  \end{align*}

  \end{theorem}

  \begin{proof}
  Let $x(t,u)$ be the solution of the differential equation 
  $$\dot{x}(t)= X(x(t))+\sum_{i=1}^{k} u_i(t)Y_i(t)+\sum_{j=1}^{l} \nu_j(t)\Xi_j(t).$$
  For $t_B>0$ consider the map $E_t:\mathcal{U} \rightarrow \mathcal{G}$  that takes $u$ to $x(t,u)$. Its derivative in some neighbourhood $N_0$ of $u=0$ is given by 
  $$ \dfrac{\partial E_t(u)}{\partial u} \biggr \vert_0=\int_0^{t} e^{s~\textnormal{ad}X} \left( \sum_{i=1}^k u_i(s)Y^i + \sum_{j=1}^l u_j(s)\Xi^j \right)  ds ~\circ~ e^{tX}. $$
  
  Denote $\frac{\partial E_t(u)}{\partial u} \big \vert_0=(F_t)_0$ and observe that this vector lies in the tangent space $T_e\mathcal{G}$. So, there must exist a vector $w$ in its dual space $T_e\mathcal{G}^\ast$ such that $\forall u \in N_0$
  $$  \big \langle w, (F_t)_0 \big\rangle =0 \iff \big \langle w_B, ((F_t)_0)_B \big\rangle =0$$
 
  which is equivalent of 

  $$  \langle w_B,  (e^{s~\textnormal{ad}X}   (Y^{i_1}+\Xi^{j_1}))_B \rangle =0 , ~ \forall s \in [0,t]$$
  Using the series expansion we obtain
  $$ \langle w_B,  (\textnormal{ad}^{i_2}(X)   Y^{i_1} + \textnormal{ad}^{j_2}(X)   \Xi^{j_1})_B  \rangle =0 , $$ for each  $i_2,j_2\geq 0, ~ 1 \leq i_1 \leq k,~1 \leq j_1 \leq l$ which contradicts the rank condition.  So, the linear map $ (F_t)_0 $ is onto in the neighbourhood $N_0$. Again, as G is a connected Lie group, it can be generated by the open subspaces i.e.,  $\cup_m N^m=\mathcal{G} $. Hence, we get 
  $$S_\Sigma (e) = G$$
  and the system is controllable.  
  \end{proof}

  \begin{corollary}
  When the Lie supergroup $\mathcal{G}$ is taken to be $\mathbb{R}_S^{m \vert n}$, then the rank condition reduces to the extended Kalman Rank condition.
  \end{corollary}
  
  \begin{proof}
  On Lie supergroup $\mathcal{G} =\mathbb{R}^{m\vert n}$, the simplest case of linear system is obtained by taking $X=Ax$ and $Y=b$. So, the linear system is $\dot{x}(t)=Ax+ub$ and the adjoints are
  \begin{align*}
  	\textnormal{ad}^0 X(Y) &=b\\
  	\textnormal{ad}^1 X(Y) &=\llbracket X,Y \rrbracket =\llbracket Ax,b \rrbracket =Ab\\
  	\textnormal{ad}^2 X(Y) &=\llbracket Ax,Ab \rrbracket =A^2b,\\
  	& \vdots
  \end{align*}
  where the Lie superbracket is given as
  $$ \llbracket X,Y\rrbracket = \frac{\partial}{\partial x} X \cdot Y - (-1)^{\vert X \vert \vert Y \vert } X \cdot \frac{\partial}{\partial x} Y. $$  
  \end{proof}

  \begin{remark}
  The above result is known as the Extended Kalman rank condition which has been established for linear systems on superspace \cite{sahoo2024controllability}.
  \end{remark}

\subsection{Examples}\label{sec4.3}

Take a control system on a Lie subsupergroup $\mathcal{G}$ of $GL(m\vert n)$ of the form
\begin{equation*}
	\dot{P}=X (P)+ \sum_{j=1}^{k} u_j Y^jP
\end{equation*} 
Here, the drift vector field $X $ defined by $X (P)=AP - (-1)^{\vert A \vert \vert P \vert} PA$ is induced by an element $A$ in the Lie superalgebra $M(m\vert n) $ and $Y^j$'s  are defined by left multiplication. Here, the flow of $X$
$$\phi_t^X (P)=e^{tA} P e^{-tA}, $$
for each $t \in \mathbb{R}^{1\vert 0}$ is an automorphism of $\mathcal{G}$.

\begin{example} 
	Consider the Lie supergroup $\mathcal{G}=SL(1\vert 1)$ which is a subgroup of $GL(1 \vert 1)$ having $sdet=1$. The corresponding Lie superalgebra $ \mathfrak{sl}(1 \vert 1)$ is generated by
	\begin{equation*}
		Y^1=  \left(\begin{array}{@{}c|c@{}}
			1 & 0  \\  \hline
			0 & 1  
		\end{array}\right), \Xi^1 =  \left(\begin{array}{@{}c|c@{}}
			0 & 1 \\  \hline
			0 & 0  
		\end{array}\right),~ \Xi^2=  \left(\begin{array}{@{}c|c@{}}
			0 & 0 \\ \hline
			1 & 0 
		\end{array}\right)
	\end{equation*}
	and satisfying the non-zero super commutation relation $\llbracket \Xi^1,\Xi^2 \rrbracket  = Y^1 $.

	Take a matrix $ X =  \left(\begin{array}{@{}c|c@{}}
		2 & 0 \\ \hline
		0 & 1 
	\end{array}\right) . $ and the $(1 \vert 0)$-parameter group of automorphism $(\phi_t^X)$ of $\mathcal{G}$ induced by $X$, i.e.,
	$$\phi_t^X(g)=e^{tX} ge^{-tX},~t\in \mathbb{R}^{1\vert 0}, g \in \mathcal{G}.$$
	
	Now consider a linear control system $\Sigma=(\mathcal{G},\mathfrak{D})$ with 
	$$ \mathfrak{D}= \left\{ X+ \nu_1 \Xi^1+\nu_2\Xi^2 \right\} $$
	
	Here, $C(\Sigma)=\{\Xi^1,~\Xi^2 \}$. So,  $\mathfrak{h}=\textnormal{Span}_{L.S.A} (C(\Sigma)) = \textnormal{Span}\{\Xi^1,~\Xi^2,~Y^1 \} $ and 
	\begin{equation*}
		\textnormal{ad}(X)(\Xi^1)=\Xi^2, ~\textnormal{ad}(X)(\Xi^2)=\Xi^1
	\end{equation*}
	gives us,
	$  \langle X \vert \mathfrak{h} \rangle = \mathfrak{sl}(2 \vert 1).$
	However,  $\textnormal{Span}\{ \Xi_1,~\Xi_2 \} \neq \mathfrak{sl}(2 \vert 1).$

	Thus, the system is transitive by LSARC but not locally controllable due to the super ad-rank condition.
\end{example}

\begin{example}
	Let us consider the Lie supergroup $\mathcal{G}=SL(2\vert 1)$ which is a subgroup of $GL(2 \vert 1)$ having $sdet=1$.

	The corresponding Lie superalgebra $\mathfrak{sl}(2 \vert 1)$ is generated by
	\begin{equation*}
		\begin{split}
			& Y^1=  \left(\begin{array}{@{}cc|c@{}}
				\frac{1}{2} & 0  & 0  \\
				0 & \frac{-1}{2} & 0 \\ [2pt] \hline
				0 & 0  & 0 
			\end{array}\right),~ Y^2=  \left(\begin{array}{@{}cc|c@{}}
				\frac{1}{2} & 0  & 0  \\
				0 & \frac{1}{2} & 0 \\ [2pt] \hline
				0 & 0  & 1 
			\end{array}\right),~Y^3=  \left(\begin{array}{@{}cc|c@{}}
				0 & 1 & 0  \\
				0 & 0 & 0 \\  \hline
				0 & 0 & 0 
			\end{array}\right),~Y^4=  \left(\begin{array}{@{}cc|c@{}}
				0 & 0 & 0  \\
				1 & 0 & 0 \\  \hline
				0 & 0 & 0 
			\end{array}\right), \\
			& \Xi^1 = \left(\begin{array}{@{}cc|c@{}}
				0 & 0 & 0  \\
				0 & 0 & 0 \\  \hline
				0 & 1 & 0 
			\end{array}\right),~ \Xi^2=\left(\begin{array}{@{}cc|c@{}}
				0 & 0 & 1  \\
				0 & 0 & 0 \\  \hline
				0 & 0 & 0 
			\end{array}\right),~ \Xi^3=\left(\begin{array}{@{}cc|c@{}}
				0 & 0 & 0  \\
				0 & 0 & 0 \\  \hline
				1 & 0 & 0 
			\end{array}\right),~\Xi^4=\left(\begin{array}{@{}cc|c@{}}
				0 & 0 & 0  \\
				0 & 0 & 1 \\  \hline
				0 & 0 & 0 
			\end{array}\right) 
		\end{split}
	\end{equation*}
	
	satisfying the following super commutation relation
	
	\[
	\begin{array}{llll}
		\llbracket Y^1, Y^3 \rrbracket=Y^3,~~&\llbracket Y^1, Y^4 \rrbracket=-Y^4,~~& \llbracket Y^1, \Xi^1 \rrbracket=\frac{1}{2} \Xi^1,~~&\llbracket Y^1, \Xi^3 \rrbracket=-\frac{1}{2} \Xi^3,\\ [4pt]
		
		\llbracket Y^1, \Xi^2 \rrbracket=\frac{1}{2} \Xi^2,~~&\llbracket Y^1, \Xi^4 \rrbracket=-\frac{1}{2} \Xi^4,~~&\llbracket Y^2, \Xi^1 \rrbracket=\frac{1}{2} \Xi^1,~~&\llbracket Y^2, \Xi^3 \rrbracket=\frac{1}{2} \Xi^3,\\ [4pt]
		
		\llbracket Y^2, \Xi^2 \rrbracket=-\frac{1}{2} \Xi^2,~~&\llbracket Y^2, \Xi^4 \rrbracket=-\frac{1}{2} \Xi^4,~~&\llbracket Y^3, \Xi^3 \rrbracket=- \Xi^1,~~&\llbracket Y^3, \Xi^4 \rrbracket=\Xi^2 ,\\ [4pt]
		
		\llbracket Y^4, \Xi^1 \rrbracket=- \Xi^3,~~&\llbracket Y^4, \Xi^2 \rrbracket= \Xi^4,~~&\llbracket Y^3, Y^4 \rrbracket=2Y^1,~~&\llbracket \Xi^1, \Xi^2 \rrbracket=Y^3,\\ [4pt]
		
		\llbracket \Xi^1, \Xi^4 \rrbracket=Y^2-Y^1,~~&\llbracket \Xi^3, \Xi^2 \rrbracket=Y^2+Y^1,~~&\llbracket \Xi^3, \Xi^4 \rrbracket=Y^4.
	\end{array}
	\]

	Now take $ X =  \left(\begin{array}{@{}cc|c@{}}
		0 & 0 & 0  \\
		1 & 0 & 0 \\ \hline
		0 & 0 & 0 
	\end{array}\right) $ and consider a linear control system $\Sigma=(\mathcal{G},\mathfrak{D})$ with 
	$$ \mathfrak{D}= \left\{ X+ u_2Y^2+\nu_1 \Xi^1+\nu_2\Xi^2 \right\}. $$

	Here, $ C(\Sigma)=\{Y^2,~\Xi^1,~\Xi^2 \}, ~ \mathfrak{h}=\textnormal{Span}_{L.S.A.}\{ C(\Sigma) \} = \textnormal{Span}\{ Y^2,~Y^3,~\Xi^1,~\Xi^2 \} $ and we have,
	\begin{align*}
		&\textnormal{ad}(X)(Y^3)=Y^1,~\textnormal{ad}(X)(\Xi^1)=\Xi^3, ~\textnormal{ad}(X)(\Xi^2)=\Xi^4,\\
		~&\textnormal{ad}^2(X)(Y^3)=Y^4.
	\end{align*}
	
	Now we get $\langle X\vert \mathfrak{h} \rangle=\textnormal{Span}_{L.S.A.}\{ C(\Sigma), ad(X)(C(\Sigma))\}=\mathfrak{sl}(2 \vert 1)$ and also $\textnormal{Span}\{ C(\Sigma), ad(X)(C(\Sigma))\}=\mathfrak{sl}(2 \vert 1)$. Thus, by the super ad-rank condition the system is locally controllable.
	
	However, if we take $ X =  \left(\begin{array}{@{}cc|c@{}}
		0 & 1 & 0  \\
		-1 & 0 & 0 \\ \hline
		0 & 0 & 0 
	\end{array}\right) $  with the same linear control system $\Sigma=(\mathcal{G},\mathfrak{D})$,	then we have,
	\begin{align*}
		&\textnormal{ad}(X)(\Xi^1)=\Xi^3, ~\textnormal{ad}(X)(\Xi^2)=\Xi^4\\
		&\textnormal{ad}(X)^2(\Xi^1)=-\Xi^1, ~\textnormal{ad}(X)^2(\Xi^2)=\Xi^2.
	\end{align*}
	
	So, we get $\langle X\vert \mathfrak{h} \rangle=\textnormal{Span}_{L.S.A.}\{ C(\Sigma), ad(X)(C(\Sigma))\}=\mathfrak{sl}(2 \vert 1)$ but  $Y^1,~Y^4 \not\in  \textnormal{Span}\{ C(\Sigma), ad(X)(C(\Sigma))\}$. Thus, the system is transitive by LSARC but not controllable by the super ad-rank condition.

\end{example}

\begin{example}
	Let us consider the Lie supergroup $\mathcal{G}=OSp(2\vert 1)$ which is a subgroup of $GL(2 \vert 1)$. The corresponding Lie superalgebra $\mathfrak{osp}(2 \vert 1)$ is generated by

	\begin{equation*}
		\begin{split}
			& Y^1=  \left(\begin{array}{@{}cc|c@{}}
				\frac{1}{2} & 0  & 0  \\
				0 & \frac{-1}{2} & 0 \\ [2pt] \hline
				0 & 0  & 0 
			\end{array}\right),~ Y^2=  \left(\begin{array}{@{}cc|c@{}}
				0 & 1 & 0  \\
				0 & 0 & 0 \\  \hline
				0 & 0 & 0 
			\end{array}\right),~Y^3=  \left(\begin{array}{@{}cc|c@{}}
				0 & 0 & 0  \\
				1 & 0 & 0 \\  \hline
				0 & 0 & 0 
			\end{array}\right) \\
			& \Xi^1 = \left(\begin{array}{@{}cc|c@{}}
				0 & 0 & \frac{1}{2}  \\
				0 & 0 & 0 \\  \hline
				0 & \frac{1}{2} & 0 
			\end{array}\right),~ \Xi^2=\left(\begin{array}{@{}cc|c@{}}
				0 & 0 & 0  \\
				0 & 0 & -\frac{1}{2} \\[2pt]  \hline 
				\frac{1}{2} & 0 & 0 
			\end{array}\right)
		\end{split}
	\end{equation*}

	satisfying the following super commutation relation

	\[
	\begin{array}{llll}
		\llbracket Y^1, Y^2 \rrbracket=Y^1 ,~~ &\llbracket Y^1, Y^3 \rrbracket=-Y^3, ~~&\llbracket Y^1, \Xi^1 \rrbracket=\frac{1}{2} \Xi^1,~~&\llbracket Y^1, \Xi^2 \rrbracket=-\frac{1}{2}\Xi^2,  \\[6pt]
		
		\llbracket Y^2, \Xi^2 \rrbracket=- \Xi^1, ~~&\llbracket Y^2, Y^3 \rrbracket=2Y^1,~~& \llbracket Y^3, \Xi^1 \rrbracket=-\Xi^2,  ~~&\llbracket \Xi^1, \Xi^2 \rrbracket=\frac{1}{2} Y^1,\\[6pt]
		
		\llbracket \Xi^1, \Xi^1 \rrbracket=\frac{1}{2} Y^2,~~&\llbracket \Xi^2, \Xi^2 \rrbracket=-\frac{1}{2} Y^3.
	\end{array}
	\]

	Take a matrix $ X =  \left(\begin{array}{@{}cc|c@{}}
		0 & 1 & 0  \\
		1 & 0 & 0 \\ \hline
		0 & 0 & 0 
	\end{array}\right) $ and consider a linear control system $\Sigma=(\mathcal{G},\mathfrak{D})$ with 
	$$ \mathfrak{D}= \left\{ X+ u_2Y^2+\nu_1 \Xi^1 \right\} $$
	Here $\mathfrak{h}=\textnormal{Span}_{L.S.A.}\{Y^2,~\Xi^1 \}=\textnormal{Span}\{Y^2,~\Xi^1 \}$ and we have,
	\begin{align*}
		\textnormal{ad}(X)(Y^2)&=-2Y^1,
		\textnormal{  ad}(X)(\Xi^1)=-\Xi^2\\
		~\textnormal{ad}^2(X)(Y^2)&=Y^3-Y^1,~\textnormal{  ad}^2(X)(\Xi^1)=\Xi^1.
	\end{align*}

	Now we get $\langle X \vert \mathfrak{h}\rangle=\mathfrak{osp}(2 \vert 1)$. Here, both the Lie superalgebraic span and the linear span cover the whole space, resulting the locally controllablility by super ad-rank condition.
	
\end{example}

\section{Conclusion}\label{sec5}
	
	In this work, a generalized version of dynamical system is given using the concept of normalizer. This proposed system generalizes different type of dynamical systems such as linear control system on both superspace and supergroup,  bilinear system on superspace and invariant system on Lie supergroup. After the characterization of normalizer, the controllability property of linear control system on Lie supergroup is studied in detail. Here, two different rank conditions are proposed, the Lie superalgebra rank condition is a necessary whereas as the ad-rank condition is a sufficient condition for the controllability of linear control system on Lie supergroup. Here, it is also shown that the proposed rank condition reduces to the extended Kalman rank condition \cite{sahoo2024controllability} when the Lie supergroup is taken to be simply the superspace.

	
	\bibliographystyle{unsrt}   
	
\end{document}